\documentclass{proc-l}
\usepackage{amssymb}
\usepackage{hyperref}
\usepackage{mathrsfs}
\hypersetup{hypertex=true,
            colorlinks=true,
            linkcolor=red,
            anchorcolor=red,
            citecolor=red}

\usepackage{tikz-cd}

\newtheorem{theorem}{Theorem}[section]
\newtheorem{lemma}[theorem]{Lemma}

\theoremstyle{definition}

\newtheorem{example}[theorem]{Example}
\newtheorem{corollary}[theorem]{Corollary}

\newtheorem{proposition}[theorem]{Proposition}
\theoremstyle{remark}
\newtheorem{remark}[theorem]{Remark}

\numberwithin{equation}{section}
\DeclareMathOperator{\Gal}{Gal}
\DeclareMathOperator{\Hom}{Hom}
\DeclareMathOperator{\Norm}{Norm}
\DeclareMathOperator{\cont}{cont}
\DeclareMathOperator{\sep}{sep}
\DeclareMathOperator{\id}{id}

\begin{document}

\title{Genus Stability of $\mathbb Z_p^\times$-Towers}


\author{Shiruo Wang}
\address{}
\curraddr{}
\email{shiruomath@gmail.com}




\dedicatory{}


\begin{abstract}
Let $p$ be a prime. Consider a tower of smooth projective geometrically irreducible curves over $\mathbb F_p$, $\mathscr C:\cdots\rightarrow C_n\rightarrow\cdots\rightarrow C_1\rightarrow C_0=\mathbb P^1$ whose Galois group is isomorphic to $\mathbb Z_p^\times$. In this paper, we study genus growth of the tower $\mathscr C$ and determine all the $\mathbb Z_p^\times$-towers with genus be a quadratic equation of $p^{n}$ when $n$ is sufficiently large.
\end{abstract}

\maketitle
\section{Introduction}
\subsection{Structure of $\mathbb Z_p^\times$-extensions}
We first set up some notation. Let $p$ be a prime number. Let $K$ be a field of characteristic $p$, and $K^{\sep}$ be its separable closure. In number theory, a common inquiry revolves around studying cyclic extensions of 
$K$ of degree $m$, denoted as $L$.

If $m$ is coprime to $p$, by Kummer theory \cite[Chapter VI]{MR1878556}, if $K$ contains a primitative $m$-th roots of unity , then such extensions can be described by solving the irreducible equation $x^m-a=0$, where $a\in K^\times$. In this case, $L/K$ is Galois with $\Gal(L/K)\cong\mathbb Z/m\mathbb Z$.

If $m=p$, then by the Artin-Schreier theory \cite[Chap X]{MR0554237}, we consider the field extension generated by a root of the irreducible equation $x^p-x=a$ for $a\in K^\times$. In this case, $L/K$ is Galois and cyclic with $\Gal(L/K)\cong\mathbb Z/p\mathbb Z.$ 

Furthermore, if $m$ is a power of $p$, the Artin-Schreier-Witt theory becomes relevant. For a field $K$, $W(K)$ denotes its Witt vector ring, $W_n(K)$ is its truncation of length $n$, and $W(K)^\times$ consists of all the invertible elements in $W(K)$.  Given $a=(a_0,a_1,\cdots)\in W_m(K)$(or $W_m(K)^\times,W_m(K^{\sep})^\times$) for $m\in\mathbb Z_{\geq1}\sqcup\{\infty\}$, the truncation map is defined to be $\pi_n(a)=(a_0,a_1,\cdots,a_{n-1})\in W_n(K)$ (or $W_m(K)^\times,W_m(K^{\sep})^\times$ respectively) for $n\leq m$. The Frobenius map $F$ on $W(K)$ (or $W_n(K),W(K)^\times,$ $W(K^{\sep})^\times$) is defined to be $F(a_0,a_1,\cdots)=(a_0^p,a_1^p,\cdots)$ (or $W_n(K),W(K)^\times, W(K^{\sep})^\times$ respectively). 

The Artin-Schreier-Witt theory \cite{MR3754335,Thomas} enables the generation of $\mathbb Z/p^n\mathbb Z$ (or $\mathbb Z_p$) extensions of $K$ by solving the Witt equation $\wp(X):=F(X)-X=a$ for some {$a\in W_n(K)/\wp W_n(K)$(or $W(K)/\wp W(K)$ respectively).

In this paper, we study $\mathbb Z_p^\times$-extensions of $K$, which parallels the Artin-Schreier-Witt theory, we call it $Kummer-Artin-Schreier-Witt~extension$. Let $\mathfrak q:W_n(K^{\sep})^\times\rightarrow W_n(K^{\sep})^\times$ be the group homomorphism defined by $\mathfrak qx:=\frac{FX}{X}$ for $X\in W_n(K^{\sep})^\times$. We consider the solutions $x=(x_0,x_1,\cdots,x_{n-1})$ to the Witt equation $\mathfrak qX=\frac{FX}{X}=a$ in $W_n(K^{\sep})^\times$ for $a=(a_0,a_1,\cdots,a_{n-1})\in W_n(K)^\times/\mathfrak qW_n(K)^\times$. 

We denote a solution of $\mathfrak qX=a$ by $\mathfrak q^{-1}a\in W(K^{\sep})^\times$, and $K(\mathfrak q^{-1}a)=K(x_0,x_1,\cdots,x_{n-1})$ represents the field extension of $K$ generated by a solution of $\mathfrak qX=a$.

\begin{theorem}\label{thm2}
     Let $n$ be a positive integer, $K$ be a field of characteristic $p>2$. Let $a=(a_0,a_1,\cdots,a_{n-1})\in W_n(K)^\times/\mathfrak qW_n(K)^\times$, then $K(\mathfrak q^{-1}a)/K$ is independent of choice and Galois. Moreover, $\Gal(K(\mathfrak q^{-1}a)/K)\cong(\mathbb Z/p^n\mathbb Z)^\times$ if and only if $a$ satisfies

    i) $a_0$ generates a subgroup of  $K^\times/\mathfrak qK^\times$ of order $p-1$,
    
    ii) $a_1\neq0$.
\end{theorem}
The case $p=2$ is slightly different from the case $p>2$ as $(\mathbb Z/2^n\mathbb Z)^\times\cong \mathbb Z/2\mathbb Z\bigoplus \mathbb Z/2^{n-2}\mathbb Z$ for $n\geq3$, but when $p>2$, $(\mathbb Z/p^n\mathbb Z)^\times\cong(\mathbb Z/p\mathbb Z)^\times\bigoplus\mathbb Z/p^{n-1}\mathbb Z$. 
\begin{theorem}\label{1.2}
    Let $K$ be a field of characteristic $p=2$. Let $a=(a_0,a_1,\cdots,a_{n-1}),b=(b_0,b_1,\cdots,b_{n-1})\in W_n(K)^\times/\mathfrak qW_n(K)^\times$ with $n\geq3$, then $K(\mathfrak q^{-1}a,\mathfrak q^{-1}b)/K$ is independent of choice and Galois. Moreover,  $\Gal(K(\mathfrak q^{-1}a,\mathfrak q^{-1}b)/K)\cong (\mathbb Z/2^{n}\mathbb Z)^\times$ if and only if $a_0=1,a_1\neq0,a_2\neq0$, $b_0=1,b_1=0,b_2\neq0$ up to the order of $a$ and $b$.
\end{theorem}
We prove Theorem \boxed{\ref{thm2}} and Theorem \boxed{\ref{1.2}} in Corollary \boxed{\ref{2.10}} and Theorem \boxed{\ref{2.17}}. 

If $a\in W_{n}(K)^\times/\mathfrak qW_{n}(K)^\times$, we denote $K(\mathfrak q^{-1}a)$ by $L_{n-1}$ and set $L_{-1}:=K$. If $\Gal(L_{n-1}/K)\cong(\mathbb Z/p^{n}\mathbb Z)^\times$, then we denote $L_{n-1}$ as $L_{[n-1]}$. When taking the projective limit of the projective system $\{\Gal(L_{[i]}/K)\}_{i=-1,0,1,\cdots}$,  the desired $\mathbb Z_p^\times$-extensions can be obtained in Corollary \boxed {\ref {cor14}}.

\subsection{A decomposition formula for 
$W(K)^\times/\mathfrak qW(K)^\times$}
In general, it is hard to solve the equation $\mathfrak qx=a$ for an arbitrary field $K$ of characteristic $p$ with $a\in W(K)^\times$. In Section \ref{sec3}, we delve into the investigation of such equations over a local field.

Let $k$ be a finite field with $p$ elements, $K=k((T))$ be a local field over $k$ with $T$ be a uniformizer. It is well known \cite[Chapter II, Proposition 5.3]{MR1697859} that we have a decomposition of $K^\times,$ 
    $$K^\times\cong \langle T\rangle\times\mu_{p-1}\times U^{(1)}$$
where $\langle T\rangle$ is a multiplicative group generated by  $T$,  and $\mu_{p-1}$ is a group of $(p-1)$-th unity in $K$, $U^{(1)}=1+Tk[\![T]\!]$.


Similar to the case in \cite{MR3754335} and \cite{MR1697859}, we find a set of generators of $W(K)^\times/\mathfrak qW(K)^\times$ as a multiplicative group. Hence we can reduce to solve $\mathfrak qX=a$ on the set of generators. More precisely, let $V$ be the Verschiebung group morphism, for $x=(x_0,x_1,\cdots)\in W(K)$, we have $V(x)=(0,x_0,x_1,\cdots)$. Let $u_i=1+V[T^i]$ for $i\leq0$. Then we can reduce $\mathfrak qX=a$ to consider equations $\mathfrak qX=[cT^d]$ and $\mathfrak qX=u_i$ for all $i\leq0$ by the following results,
\begin{theorem}{\label {13}}
     There exists a unique representation of $x\in W(K)^\times/\mathfrak qW(K)^\times$. 
     
(Corollary \boxed{\ref{3.8}}) If $p>2$, then
$$x\equiv [cT^d]\cdot\prod_{i}u_i^{k_i}\bmod\mathfrak qW(K)^\times$$
with $c\in k^\times, (i,p)=1$ or $i=0$, $0< d\leq p-1,k_i\in\mathbb Z_p$ and $k_i\rightarrow 0$ $p$-adically as $i\rightarrow-\infty$.

(Corollary \boxed{\ref{3.7}}) If $p=2$, let $v_i=1+V^2[T^i]$. Then 
$$x\equiv\prod_{i}u_i^{k_i}\cdot v_i^{k_i'}\bmod\mathfrak qW(K)^\times$$
with $(i,p)=1$ or $i=0$, $k_i,k_i'\in\mathbb Z_2$ and $k_i,k_i'\rightarrow0$ $2$-adically as $i\rightarrow-\infty.$
\end{theorem} 

Hence for an equation $\mathfrak qx=a$, one can decompose $a\in W(K)^\times/\mathfrak qW(K)^\times$ as above and reduce the equation into the case $a$ is an element on the set of generators.


\subsection{Genus growth of $\mathbb Z_p^\times$-towers}
We define the $\mathbb Z_p^\times$-towers for the projective line $\mathbb P^1$ associated to $(a_0,a_1,\cdots)\in W(K)^\times$ as follows,
$$\mathscr C:\cdots\rightarrow C_n\rightarrow\cdots\rightarrow C_0\rightarrow C_{-1}=\mathbb P^1.$$
For each $n$, the curve $C_n$ is defined by the solution of the equation $$\mathfrak q(y_0,y_1,\cdots,y_{n})=(a_0,a_1,\cdots,a_{n})$$ for some $(y_0,y_1,\cdots)\in W(K^{\sep})^\times$ and $\Gal(C_n/C_{-1})\cong(\mathbb Z/p^{n+1}\mathbb Z)^\times$. In this paper, we assume the tower is totally ramified at $\infty$ (otherwise one can change coordinate) and unramified at all other points. We will also consider a tower ramified at finitely many points by localization in Section \ref{sec5}.

We call such a tower is a $geometric$ tower if the field subextension corresponding to $C_{i+1}/C_i$ does not contain a constant extension for all $i\geq-1$. In Sections \ref{sec4} and \ref{sec5}, we study the genus growth properties of $\mathscr C$. In general, estimating the genus of $g(C_n)$ when $n$ is large  for a general cover can be challenging, but if $\Gal(\mathcal C)$ isomorphic to a quotient group of $\mathbb Z_p^\times$, we prove the following theorem in Corollary \boxed{\ref{4.18}}
\begin{theorem}\label{1.5}
    Let $\mathcal C$ be a geometric $\mathbb Z_p^\times$-tower defined as above generated by the solutions of $\mathfrak qy=x$ for some $x=[cT^d]\cdot\prod_iu_i^{k_i}$ with $u_i=1+V[T^i]$ as in Corollary \boxed{\ref{3.8}}. Then the followings are equivalent
    
    i) There exists $a,b,c\in\mathbb Q$ such that for all $N\in\mathbb Z_{>0},$ if $n>N$, then $$g(C_n)=ap^{2n}+bp^n+c.$$
    
    ii) There exists $s\in\mathbb Z_{>0}$ with $(s,p)=1$ and $m,t\in\mathbb Z_{\geq0}$ satisfy $s\equiv t\bmod (p-1)$ such that for $n$ large enough, $$\max\{ip^{-v_p(k_i)}:(i,p)=1\}=\frac{1}{p-1}(\frac{s}{p^m}-\frac{t}{p^{n-1}}).$$
\end{theorem}

This paper is structured as follows. In Section \ref{sec2}, we discuss the group structure of $W(K)^\times/\mathfrak qW(K)^\times$. In Section \ref{sec3}, we fix a set of generators of $W(K)^\times/\mathfrak qW(K)^\times$ and prove a decomposition theorem for the elements in $W(K)^\times/\mathfrak qW(K)^\times$. In Section \ref{sec4}, we compute the genus growth formula for local field and generate it into function field case in Section \ref{sec5}.

\textbf{Acknowledgements:} The author thanks Hui June Zhu for some help discussions on this project during his graduate study in University at Buffalo. The paper study the case first suggested in \cite[Remark 3.3]{kosters2016arithmeticzpextensions}.

\section{Kummer-Artin-Schreier-Witt extensions}\label{sec2}

In this section, we will prove Theorem \boxed{\ref{thm2}}.

Fix $p$ be a prime, let $K$ be a field of characteristic $p$. Recall the Kummer theory studies the $\mathbb Z/m\mathbb Z$-extensions for some $m$ coprime to $p$ and Artin-Schreier theory studies the $\mathbb Z/p\mathbb Z$-extensions. Now we will develop these two theories via the Witt vector to study the $\mathbb Z_p^\times$-extensions of $K$. Let $L_{[n]}$ be a nondegenerate field extension of $K$ satisfies $\Gal(L_{[n]}/K)\cong(\mathbb Z/p^{n+1}\mathbb Z)^\times$ and $\Gal(L_{[\infty]}/K)\cong\mathbb Z_p^\times$. 

We separate this section into three parts.  In Section \ref{3.1} and Section \ref{3.2}, we discuss case $p>2$ and case $p=2$ separately. In Section \ref{s3.3} we endow a topology structure for the Galois group when $n\rightarrow\infty$, which will give us the structure of $\mathbb Z_p^\times$-extension of $K$. The main results in this section are the structure of $L_{[n]}$ for $p>2$ in Corollary \boxed{\ref{2.10}} and $p=2$ in Theorem \boxed{\ref{2.17}}. By passing the limit, we have these two results hold for $n\rightarrow\infty$ in Corollary $\boxed{\ref{cor14}}.$   

Recall $W_n(K^{\sep})^\times$ is the collection of all invertible elements in $W_n(K^{\sep})$. Recall $F$ is the Frobenius map with $F(x)=x^p$, we define $\mathfrak q$ on $W_n(K^{\sep})^\times$ as
\begin{eqnarray*}
\mathfrak q=\frac{F}{id}: W_n(K^{\sep})^\times&\rightarrow& W_n(K^{\sep})^\times\\
x&\mapsto&F(x)\cdot y
\end{eqnarray*}
where $y\in W_n(K^{\sep})^\times$ satisfies $xy=(1,0,0,\cdots,0)\in W_n(K^{\sep})^\times$. 

Recall an element $(a_0,a_1,\cdots)$ in $W_n(K)$ is a unit if and only if $a_0\neq0$. $\mathfrak q$ is a group homomorphism of $W_n(K^{\sep})^\times$ with $\ker\mathfrak q=W_n(\mathbb F_p)^\times$. For $a\in W_n(K)^\times$, if there exists $y=(y_0,y_1,\cdots,y_{n-1})\in W_n(K^{\sep})^\times$ satisfies $\mathfrak qy=a$, then we write $y=\mathfrak q^{-1}a$ and denote the field extension $K(y_0,y_1,\cdots,y_{n-1})$ as $K(\mathfrak q^{-1}a)$. One can show $K(\mathfrak q^{-1}a)= K(\mathfrak q^{-1}b)$ if $a\equiv b\bmod \mathfrak qW_n(K)^\times$ and by induction one has $\mathfrak q^{-1}a\in W_n(K^{\sep})^\times$ for every element $a\in W_n(K)^\times$. Notice $F$ commutes with $W_n(K)^\times$-embedding $W_n( K(\mathfrak q^{-1}a))^\times\hookrightarrow W_n(K^{\sep})^\times$ induced by $K(\mathfrak q^{-1}a)\hookrightarrow {K^{\sep}}$, then $K(\mathfrak q^{-1}a)/K$ is Galois. 

\begin{remark}
    In general, we need to consider the field extension generated by the collection of all the roots of the equation $\mathfrak qy=a$. However the extension does not depend on the choice of solution as it is Galois.
\end{remark}

If $a\in W_n(K)^\times$, set 
\begin{equation*}\tag{*}
r={\left\{
\begin{aligned}
    \min\limits_{m\in\{1,2,\cdots,n-1\}}&\{m: a_m\neq0\}&~\text{if}~~a\neq(a_0,0,0,\cdots,0),\\
&n&\text{if} ~a=(a_0,0,0,\cdots,0)
\end{aligned}
\right.}
\end{equation*}

and $s=n-r$. 
\subsection{Structure of $(\mathbb Z/p^n\mathbb Z)^\times$-extensions for $p>2$}\label{3.1}
\begin{lemma}\label{thm-1}
If $p>2$, $a=(a_0,0,\cdots,0,a_r,a_{r+1},\cdots,a_{n-1})\in W_n(K)^\times/\mathfrak qW_n(K)^\times$ for some $0\leq r< n$ with $n\in\mathbb Z_{\geq1}\sqcup\{\infty\},a_r\neq0$, then $a^p=(a_0',0,\cdots,0,a_{r+1}',\cdots,a_{n-1}')\in W_n(K)^\times/\mathfrak qW_n(K)^\times,$ here $a_0'=a_0^p\neq0$ and $a_{r+1}'\neq0$.

\end{lemma}
\begin{proof}
Suffice to show for all $a$, we have $(\pi_{r+2}(a^p))=(\pi_{r+2}(a))^p=(a_0,0,\cdots,0,a_r,a_{r+1})^p=(a_0^p,0,\cdots,0,a_{r+1}')$ and $a_{r+1}'\neq0$. The lemma follows from the fact that the $(r+2)$ -th ghost vector of $a^p$ is $(a^p)^{(r+2)}=(a_0^p)^{p^{r+1}}+p^{r+1}a_r^pa_0^{p^{r+1}-p^r}+p^{r+2}\cdot u$ for some $u\in\mathbb Z[a_0,a_r,a_{r+1}]$.

\end{proof}

\begin{proposition}\label{thm1}
The order of $a\in W_n(K)^\times/\mathfrak qW_n(K)^\times$ divides $p^{s}(p-1)$ and is equal to $p^s(p-1)$ if and only if $\langle a_0\rangle$ is a subgroup of $K^\times/\mathfrak qK^\times$ of order $p-1$.
\end{proposition}
\begin{proof}
By Lemma \boxed{\ref {thm-1}}.
\end{proof}
Let $L/K$ be a finite Galois extension, $g\in \Gal(L/K)$ induces the map $g:W_n(L)\rightarrow W_n(L)$ and we can define the norm map
\begin{eqnarray*}
\Norm_{W_n(L)/W_n(K)}:W_n(L)&\rightarrow& W_n(K)\\
x&\mapsto&\Pi_{g\in \Gal(L/K)}g(x).
\end{eqnarray*}
Then we have the following theorem:
\begin{theorem}[Hilbert 90]
    Let $L/K$ be a cyclic extension of degree $m>0$ with Galois group $\Gal(L/K)=\langle\sigma\rangle$, then there is a short exact sequence 
    $$1\rightarrow W_n(L)^\times/W_n(K)^\times\xrightarrow{\mathfrak q} W_n(L)^\times\xrightarrow{\Norm} W_n(K)^\times\rightarrow 1.$$
\end{theorem}

Let $y=(y_0,y_1,\cdots,y_{n-1})\in W_n(K^{\sep})^\times$ satisfies $\mathfrak qy=a$ for some $a\in W_n(K)^\times$. 

\begin{theorem}{\label{038}}

     When $p>2$, let $a\in W_n(K)^\times/\mathfrak q W_n(K)^\times$. Then the maps
\begin{eqnarray*}
\psi_1:\Gal(K(\mathfrak q^{-1}a)/K)&\rightarrow& \Hom(\langle a\rangle,W_n(\mathbb F_p)^\times)\\
g&\mapsto&(\mathfrak q y\mapsto \frac{g(y)}{y})
\end{eqnarray*}
and
\begin{eqnarray*}
\psi_2:\langle a\rangle&\rightarrow& \Hom(\Gal(K(\mathfrak q^{-1}a)/K),W_n(\mathbb F_p)^\times)\\
\mathfrak qy&\mapsto&(g\mapsto \frac{g(y)}{y})
\end{eqnarray*}
are isomorphisms.

Moreover, we have the following isomorphisms
$$(\mathbb Z/p^{s+1}\mathbb Z)^\times\cong \Hom(\langle a\rangle,W_n(\mathbb F_p)^\times)\cong\langle a\rangle\cong \Hom(\Gal(K(\mathfrak q^{-1}a)/K),W_n(\mathbb F_p)^\times)$$
with $s$ as defined in $(*)$.
\end{theorem}
\begin{proof}
For vector $1\in W_n(K)^\times$, we have $$\Norm_{W_n(K(\mathfrak q^{-1}a))/W_n(K)}1=\Pi_{g\in \Gal(K(\mathfrak q^{-1}a)/K)}g(1)=1.$$ 
By Hilbert 90, there exists $y=(y_0,y_1,\cdots,y_{n-1})\in W_n(K(\mathfrak q^{-1}a))^\times$ satisfies $1=\frac{y}{\sigma y}$. Moreover,

$$\sigma\mathfrak q y=\sigma(\frac{F(y)}{y})
=\frac{\sigma(F(y))}{\sigma y}
=\frac{F(\sigma y)}{\sigma y}
=\frac{F(y)}{y}
=\mathfrak qy.$$

Therefore $\mathfrak qy\in W_n(K)^\times$ and $K(y)=K(\mathfrak q^{-1}a)$.
Pick  $g_1, g_2\in \Gal(K(\mathfrak q^{-1}a)/K)$, 
$\psi_1(g_1)=\psi_1(g_2)$ if and only if $g_1(y)=g_2(y)$, so $\psi_1$ is a well-defined injection. On $W_n(K)^\times/\mathfrak qW_n(K)^\times$, by Proposition \boxed{\ref{thm1}} we have the order of $a$ is $p^{s}(p-1)$ and we claim that $$\langle a\rangle\cong \Hom(\langle a\rangle, W_{n}(\mathbb F_p)^\times)\cong(\mathbb Z/p^{s+1}\mathbb Z)^\times.$$  When $n=r$, then $a=(a_0,0,0,\cdots,0)$, the group generated by  $a_0$ has the same order with the group generated by $a$, the claim
$$\langle a\rangle\cong \Hom(\langle a\rangle, W_{n}(\mathbb F_p)^\times)\cong(\mathbb Z/p\mathbb Z)^\times$$ 
follows from the fact $$\langle a_0\rangle=\{1,a_0,a_0^2,\cdots,a_0^{p-1}\}\cong \Gal(K(\mathfrak q^{-1}a_0)/K)\cong \Hom(\langle a_0\rangle,W_1(\mathbb F_p)^\times).$$

Assume $r<n$. Consider the following commutative diagram
$$\begin{tikzcd}
  G=\Gal(K(\mathfrak q^{-1}a)/K) \arrow[r, hook] \arrow[d, two heads]
    & \Hom(\langle a\rangle,W_n(\mathbb F_p)^\times)\cong(\mathbb Z/p^{s+1}\mathbb Z)^\times \arrow[d, two heads ] \\
  G_{r+1}:=\Gal(K(\mathfrak q^{-1}\pi_{r+1}(a))/K) \arrow[r, hook ]
& |[]| \Hom(\langle \pi_{r+1}(a)\rangle,W_{r+1}(\mathbb F_p)^\times)\cong (\mathbb Z/p^2\mathbb Z)^\times.
\end{tikzcd}$$
We first show that the bottom map is an isomorphism. By Proposition \boxed{\ref{thm1}}, $\Hom(\langle\pi_{r+1}(a),W_{r+1}(\mathbb F_p)^\times)\cong(\mathbb Z/p^2\mathbb Z)^\times$. For the left hand side part, write $\mathfrak q^{-1}\pi_{r+1}(a)$ as $(y_0,0,0,\cdots,y_r)$, then $K(y_0, y_r)/K(y_0)$ is a field extension demonstrated by an Artin-Schreier equation, so $[K(y_0,y_r):K(y_0)]=p$. Hence we have $G_{r+1}\cong(\mathbb Z/p^2\mathbb Z)^\times.$ 

Since the left vertical map can be described by the projection map $\pi_{r+1}$ and the second vertical map is natural projection, so both are surjection. Since $G_{r+1}\cong \Hom(\langle \pi_{r+1}(a)\rangle,W_{r+1}(\mathbb F_p)^\times)$, we have $\psi_1:G\rightarrow \Hom(\langle a\rangle,W_n(\mathbb F_p)^\times)$ is an isomorphism.

One may endow a discrete topology on $G_r$ for each $r$, so $G_r$ is a locally compact abelian group. 
Notice that the map $\psi_2$ is dual of $\psi_1$, we have $\psi_2$ is also an isomorphism by Pontryagin duality theorem \cite[Chapter 4]{morris_1977}.
\end{proof}
\begin{corollary}\label{2.10}
    If $p>2$, then $\Gal(K(\mathfrak q^{-1}a)/K)\cong(\mathbb Z/p^n\mathbb Z)^\times$ if and only if $a=(a_0,a_1,\cdots,a_{n-1})$ satisfies

    i) $a_0$ generates a subgroup of $K^\times/\mathfrak qK^\times$ of degree $p-1$,
    
    ii) $a_1\neq0$.

\end{corollary}
\begin{proof}
    Let $s=n-1$ in Theorem \boxed{\ref{038}}. By the right vertical map in the commutative diagram, $\langle\pi_2(a)\rangle$ has order $p(p-1)$, hence $a_1\neq0$. By the first horizon map in the above commutative diagram, one has $\Gal(K(\mathfrak q^{-1}a)/K)\cong(\mathbb Z/p^n\mathbb Z)^\times$. 
\end{proof} 

From now on, we denote $K(\mathfrak q^{-1}a)$ as $L_{[n]}$ if $\Gal(K(\mathfrak q^{-1}a)/K)\cong(\mathbb Z/p^{n+1}\mathbb Z)^\times$ and $L_{-1}:=K$.
\subsection{Structure of $(\mathbb Z/p^n\mathbb Z)^\times$-extensions for $p=2$}\label{3.2}
Now we consider the case where $p=2$ and we will prove a similar structure theorem as in Corollary \boxed{\ref{2.10}}.

\begin{lemma}{\label{lemma9}}
For $1\leq m\leq 2^n$, we have $v_2(\binom{2^n}{m})=n-c-1$ where $c=\max\{k\in\mathbb Z:2^k\mid m\}$.
\end{lemma}
\begin{proof}
One can easily show it by writing $m$ in 2-adic form and by the fact $v_p(\binom{n}{m})=\frac{S_p(m)+S_p(n-m)-S_p(n)}{p-1}$ where $S_p(m)$ is the sum of the $p$-adic digits of $m$.
\end{proof}


\begin{lemma}{\label{312}}
    Recall $r$ is defined in $(*)$. For $a\in W_n(K)^\times/\mathfrak qW_n(K)^\times$, if $r=1$ and $a_2\neq0$, then $a^2=(1,0,0,a_3',a_4',\cdots,a_{n-1}')$ with $a_3'\neq0$. If $r\geq2$, then $a^2=(1,0,\cdots,a_{r+1}',\cdots,a_{n-1}')$  with $a_{r+1}'\neq0$.  
\end{lemma}
\begin{proof}
For the case $r=1$ and $a_2\neq0$, one can consider $\pi_4(a^2)$. For the case $r\geq2$, one can consider $\pi_{r+2}(a^2)$.
\end{proof}
\begin{theorem}{\label{2.17}}
When $p=2$ and $n\geq 3$, $L_{[n-1]}$ is the compositum of $K_1,K_2$,
where $K_1=K(\mathfrak q^{-1}(a_0,\cdots,a_{n-1})),$ $K_2=K(\mathfrak q^{-1}(b_0,\cdots,b_{n-1}))$ if and only if $(a_0,a_1,\cdots,a_{n-1})$, $(b_0,b_1,\cdots,b_{n-1})\in W_n(K)^\times/\mathfrak qW_n(K)^\times$ with $a_0=1,a_1\neq0,a_2\neq0$ and $b_0=1,b_1=0,b_2\neq0$.

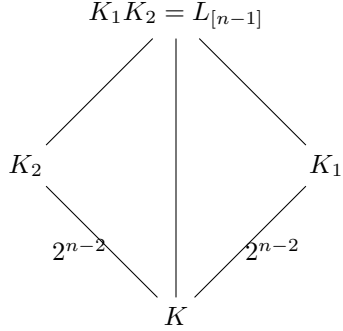
\center\begin{tikzpicture}
    \node (Q1) at (0,0) {$K$};
    \node (Q2) at (2,2) {$K_1$};
    \node (Q3) at (0,4) {$K_1K_2=L_{[n-1]}$};
    \node (Q4) at (-2,2) {$K_2$};
    \draw (Q1)--(Q2) node [pos=0.7, below,inner sep=0.25cm] {$2^{n-2}$};
    \draw (Q1)--(Q4) node [pos=0.7, below,inner sep=0.25cm] {$2^{n-2}$};
    \draw (Q3)--(Q4);
    \draw (Q2)--(Q3);
    \draw (Q1)--(Q3) node [pos=0.7, inner sep=0.25cm]{};
\end{tikzpicture}

\end{theorem}
\begin{proof}
$a_0=b_0=1$ comes from $\mathfrak qa_0=\mathfrak qb_0=1$. On one hand, if $b=(1,0,b_2,\cdots,b_{n-1})$ for some $b_2\neq0$, then by the commutative diagram in Theorem \boxed{\ref{038}}, one has $\Gal(K(\mathfrak q^{-1}(1,0,b_2))/K)\cong(\mathbb Z/4\mathbb Z)^\times$ and hence $\Gal(K_2/K)\cong\mathbb Z/2^{n-2}\mathbb Z$.

On the other hand, if $a=(1,a_1,\cdots,a_{n-1})$ for some $a_1,a_2\neq0$, then by the $n=2$ case and Lemma \boxed{\ref{312}}, one has $\Gal(K(\mathfrak q^{-1}(1,a_1,a_2))/K)$ is a cyclic group of order $2$ and hence $\Gal(K_1/K)\cong\mathbb Z/2^{n-2}\mathbb Z$.

Moreover, by the computation above, one can easily check $b$ cannot be generated by $a$ since the third component of the even power of $a$ is always 0 and the second component in the odd power of $a$ is always nonzero. 

If we have $a'=(1,0,0,a'_3,\cdots,a'_{n-1})\in W_n(K)^\times/\mathfrak qW_n(K)^\times$, then by Lemma \boxed{\ref{2.17}}, the order of the group generated by $a'$ is less than $2^{n-1}$. By the result in Theorem \boxed{\ref{038}}, the order of $\Gal(K(\mathfrak q^{-1}a)/K)$ is also less than $2^{n-1}$, hence we have the desired if and only if condition.
\end{proof}
\subsection{Structure of $\mathbb Z_p^\times$-extensions}\label{s3.3}
In this part, we prove the results in the previous also hold for $n\rightarrow\infty$ by taking the projective limit after endowing a proper topology for $\Gal(K(\mathfrak q^{-1}a/K))$ and $W_n(K)^\times/\mathfrak qW_n(K)^\times$.
\begin{theorem}\label{thm3}
Let $K$ be a field with character $p>2$ and $K^{ab}$ be the maximum abelian extension of $K$, $G=\Gal(K^{ab}/K)$. For each fixed $n\in\mathbb Z_{\geq1}$, we have the following perfect pairing 
\begin{eqnarray*}
\Phi:G/G^{p^{n-1}(p-1)}\times W_n(K)^\times/\mathfrak qW_{n}(K)^\times&\rightarrow& W_n(\mathbb F_p)^\times\\
(g,\mathfrak q b\bmod {\mathfrak qW_n(K)^\times})&\mapsto&\frac{gb}{b}
\end{eqnarray*}
and the homeomorphisms between topological groups
\begin{eqnarray*}
\phi_1:G/G^{p^{n-1}(p-1)}&\rightarrow&\Hom(W_n(K)^\times/\mathfrak qW_n(K)^\times,W_n(\mathbb F_p)^\times)\\
g&\mapsto&(\mathfrak qb\bmod\mathfrak qW_n(K)^\times\mapsto\frac{gb}{b})\\
\phi_2:W_n(K)^\times/\mathfrak qW_n(K)^\times&\rightarrow&\Hom_{\cont}(G/G^{p^{n-1}(p-1)},W_n(\mathbb F_p)^\times)\\
\mathfrak qb\bmod\mathfrak qW_n(K)^\times&\mapsto&(g\mapsto\frac{gb}{b}).
\end{eqnarray*}
\end{theorem}
\begin{proof}
First, we can easily check $\Phi$ is well-defined. Notice that the left kernel of $\Phi$ is $\id$ and the right kernel of $\Phi$ is $W_n(K)^\times$, we have $\Phi$ is nondegenerate.

Now we prove that $\phi_2$ is a homeomorphism. 
Let $\chi\in \Hom_{\cont}(G/G^{p^{n-1}(p-1)},W_n(\mathbb F_p)^\times)$, by Theorem \boxed{\ref{038}}, the order of the subgroup generated by $\mathfrak qb\in W_n(K)^\times/\mathfrak qW_n(K)^\times$ is a factor of $p^{n-1}(p-1)$, hence one can find $\mathfrak qb\in W_n(K)^\times/\mathfrak qW_n(K)^\times$ such that $\chi$ factors through the following commutative diagram:
$$\begin{tikzcd}
  G/G^{p^{n-1}(p-1)} \arrow[r, twoheadrightarrow] \arrow[dr, rightarrow]
    & \Gal(K( b)/K) \arrow[d]\\
&W_n(\mathbb F_p)^\times \end{tikzcd}$$
and $\langle \mathfrak qb\rangle\cong\Hom(\Gal(K(b)/K),W_n(\mathbb F_p)^\times)$. Hence for each $\chi$ we can find a multiple of $\mathfrak qb$ such that its image on the map of $\phi_2$ is $\chi$. Since we endow $W_n(K)^\times/\mathfrak qW_n(K)^\times$ with discrete topology, we consider an open basis around $1$ given by $\{1+(W_n(K)^\times/\mathfrak qW_n(K)^\times)^{p^i}:i\geq0\}_i$ for $W_n(K)^\times/\mathfrak qW_n(K)^\times$ and the $p$-adic topology for $W_n(\mathbb F_p)^\times$, the map $\phi_2$ is a homeomorphism. By the Pontryagin duality theorem, $\phi_1$ is also a homeomorphism. 
\end{proof}
\begin{corollary}{\label {cor14}}
Let $G_p^\times=\varprojlim\limits_nG/G^{p^{n-1}(p-1)}:=G/G^{p-1}\bigoplus\varprojlim\limits_nG/G^{p^{n}}$, we have the following perfect pairing 
\begin{eqnarray*}
\Phi:G_p^\times\times W(K)^\times/\mathfrak qW(K)^\times&\rightarrow& W(\mathbb F_p)^\times\\
(g,\mathfrak q b(\bmod \mathfrak qW(K)^\times))&\mapsto&\frac{gb}{b}.
\end{eqnarray*}
\end{corollary}
\begin{proof}
Notice $W_n(K)^\times/\mathfrak qW_n(K)^\times$ forms a projective system with an open basis around 1 given by $\{1+(W_n(K)^\times/\mathfrak qW_n(K)^\times)^{p^i}:i\geq0\}_i$, which is equivalent to the classical $p$-adic topology. We endow the nature Krull topology for $\Gal(L_{n}/K)$, which is the same as the discrete topology when $n$ is finite and the $p$-adic topology for $W_n(\mathbb F_p)^\times$.  Taking the projective limit on $\phi_2$ in Theorem \boxed{\ref{thm3}}, one has a homeomorphism between $W(\mathbb F_p)^\times$ and $\Hom(G_p^\times,W(K)^\times/\mathfrak qW(K)^\times)$.
\end{proof}
\section{A decomposition formula on Witt ring of local field}\label{sec3}

In this section, we will prove Theorem \boxed{\ref{13}}.

Let $k$ be a finite field of characteristic $p$ and $K=k(\!(T)\!)$ be a local field over $k$ with $T$ an indeterminant. For an element $f=\sum_ia_iT^i\in K$, we denote $v_T(f)=v$ if $a_v\neq 0$ and $a_i=0$ for all $i<v$. Especially, $v_T(T)=1$ and we call $T$ a uniformizer of $K$. In this section, we will discuss the structure of $W(K)^\times/\mathfrak qW(K)^\times$ explicitly by a decomposition formula in Corollary \boxed{\ref{3.7}} and we will use this formula in Section \ref{sec4} and Section \ref{sec5}.

Recall $V$ is the Verschiebung map over $W(K)$, says $V(a_,a_1,\cdots)=(0,a_0,a_1,\cdots)$. Let $u_{m,n}=1+V^{n-1}[T^m]\in W(K)^\times$, that is $\pi_n(u_{m,n})=(1,0,0,\cdots,T^m)\in W_n(K)^\times,n\geq2$. 

\begin{lemma}{\label{034}}
Let $x=(1,0,\cdots,f(T))\in W_n(K)^\times$ with $f(T)=a_{-m}T^{-m}+\cdots+a_0+a_1T+\cdots+a_{m'}T^{m'}$, then $x$ can be written as $\pi_n(\prod_{i=-m}^{m'} u_{i,n}^{a_i})$.
\end{lemma}
\begin{proof}
We first consider the case $x\in W_2(K)^\times$. Notice the $n$-th ghost component of $u_{i,2}$ is $u_{i,2}^{(n)}=1+pT^{ip^{n-2}}$, then $(u_{i,2}\cdot u_{j,2})^{(n)}=1+p(T^{ip^{n-2}}+T^{jp^{n-2}})+p^2\cdot T^{(i+j)p^{n-2}}$. Hence $\pi_2(u_{i,2}\cdot u_{j,2})=(1,T^i+T^j)$, $\pi_2(u_{i,2}^{k_i})=(1,k_iT^i)$ for $k_i=1,2,\cdots,p-1$ and one has $x=\pi_2(\prod_{i=-m}^{m'} u_{i,2}^{a_i})$. Let $x=x=1+V^{n-1}[f(T)]$ one has the desired result.
\end{proof}
Now we have a decomposition theorem for $W(K)^\times$, which generates the structure theorem in \cite[Chapter II, Proposition 5.3]{MR1697859}.
\begin{theorem}{\label{037}}
Let $p$ be a prime. Let $x=(f_0(T),f_1(T),\cdots)\in W(K)^\times$, with $f_i(T)=a_{-m,i}T^{-m}+a_{-m+1,i}T^{-m+1}+\cdots+a_{0,i}+a_{1,i}T+\cdots$ . Then 
$$x=[cT^d]\cdot\prod_{k=0}^\infty\prod_{(j,p)=1}[1-c_{jk}T^j]^{p^k}\cdot\prod_{i_2=2}^\infty\prod_{i_1>-\infty}u_{i_1,i_2}^{k_{i_1,i_2}}$$
with $c\in k^\times$, $d=v_T(f_0(T))$, $0\leq k_{i_1,i_2}\leq p-1,$ and $[\quad]$ is the Teichm\"{u}ller lifting.  
\end{theorem}
\begin{proof}
We use induction to prove the statement. The factorization of $f_0(T)=cT^d\cdot\prod_{k=1}^\infty\prod_{(j,p)=1}(1-c_{jk}T^j)^{p^k}$ comes from the structure theorem of $K^\times$ \cite[Chapter II, Proposition 5.3]{MR1697859}. Assume $\pi_{n}(x)$ can be written in the same form as in the statement. Let 
$$\pi_{n}(x)=\pi_{n}([cT^d]\cdot\prod_{k=0}^\infty\prod_{(j,p)=1,j>0}[1-c_{jk}T^j]^{p^k}\cdot\prod_{i_1}\prod_{i_2=2}^nu_{i_1,i_2}^{k_{i_1,i_2}}),$$ 
write 
$$\pi_{n+1}([cT^d]\cdot\prod_{k=0}^\infty\prod_{(j,p)=1}[1-c_{jk}T^j]^{p^k}\cdot\prod_{i_1}\prod_{i_2=2}^nu_{i_1,i_2}^{k_{i_1,i_2}})=(f_0(T),f_1(T),\cdots,f_{n-1}(T),g_{n}(T)).$$

Suffice to show $$\pi_{n+1}(x)=(f_0(T),f_1(T),\cdots,f_{n-1}(T),g_{n}(T))\cdot\prod_{i_1\geq n'}u_{i_1,n+1}^{k_{i_1,n+1}}$$ for  $k_{i_1,n+1}=0,1,\cdots,p-1$ and $n'$ is some constant depends on $g_n(T)$ and $f_n(T)$.

Write $\pi_{n+1}(\prod_{i_1\geq n'}u_{i_1,n+1}^{k_{i_1,n+1}})=(1,0,0,\cdots,\sum_{i_1=n'}^\infty k_{i_1,n+1}T^{i_1})$, then we consider the equation
\begin{eqnarray*}
&\pi_{n+1}((f_0(T),f_1(T),\cdots,f_{n-1}(T),g_{n}(T))\cdot\prod_{i_1}u_{i_1,n+1}^{k_{i_1,n+1}})\\
&=(f_0(T),\cdots,f_{n-1}(T),g_{n}(T)+f_0(T)^{p^n}(\sum_{i_1=n'}^\infty k_{i_1,n+1}T^{i_1})).
\end{eqnarray*}
One can find a unique $k_{i_1,n+1}\in k$ by comparing the coefficients in $f_n(T)$ and $g_{n}(T)+f_0(T)^{p^n}\cdot(\sum_{i_1=n'}^\infty k_{i_1,n+1}T^{i_1})$. Hence one can construct the decomposition for $\pi_{n+1}(x)$ from $\pi_{n}(x)$ by multiplying appropriate $u_{i_1,n+1}$.

When $n\rightarrow\infty$, one can define the infinite product $\prod_{i_1,i_2}u_{i_1,i_2}^{k_{i_1,i_2}}$ componentwise $T$-adically converges in $W(K)^\times$ .
\end{proof}
Now we separate the case with $p>2$ and $p=2$ in the following two corollaryies.
\begin{corollary}{\label{3.8}}
    If $p>2$, let $u_i=1+V[T^i]$. There exists a unique representation of $x\in W(K)^\times/\mathfrak qW(K)^\times$,
$$x\equiv [cT^d]\cdot\prod_{i}u_i^{k_i}\bmod\mathfrak qW(K)^\times$$
with $c\in k^\times, (i,p)=1$ or $i=0$, $0< d\leq p-1,k_i\in\mathbb Z_p$ and $v_p(k_i)\rightarrow\infty$ as $i\rightarrow\infty$.
\end{corollary}
\begin{proof}
    Notice $a^p\equiv a\bmod qW(K)^\times$ for all $a\in W(K)^\times$ and by Theorem \boxed{\ref{037}}, we have $$x\equiv[T^{pm}]\cdot[cT^d]\cdot\prod_{k=0}^\infty\prod_{(j,p)=1}[1-c_{jk}T^j]^{p^k}\cdot\prod_{i_1,i_2}u_{i_1,i_2}^{k_{i_1,i_2}}\bmod \mathfrak qW(K)^\times$$
    for some proper $m$ such that $0< d\leq p-1$. Moreover, $[1-c_{jk}T^j]^{p^k}\equiv[1-c_{jk}T^j]\bmod\mathfrak qW(K)^\times$. Hence we have $(i_1,p)=1$ and such representation in $W_n(K)^\times/\mathfrak qW_n(K)^\times$ is unique. Notice that since $j>0$, by Hensel's lemma \cite[Chapter 1, Theorem 3]{Koblitz1984}, the corollary follows from $\pi_{m+2}(u_i^{p^m})\equiv u_{i,m+2}\bmod\mathfrak qW_{m+1}(K)^\times$ for all $m\in\mathbb Z$.
\end{proof}
\begin{corollary}\label{3.7}
    For $p=2$, let $u_i=1+V[T^i],v_i=1+V^2[T^i]$. There exists a unique representation of $x\in W(K)^\times/\mathfrak qW(K)^\times$,
$$x\equiv\prod_{i}u_i^{k_i}\cdot v_i^{k_i'}\bmod\mathfrak qW(K)^\times$$
with $(i,p)=1$ or $i=0$, $k_i,k_i'\in\mathbb Z_2$ with $v_2(k_i),v_2(k_i')\rightarrow\infty$ as $i\rightarrow\infty$.
\end{corollary}
\begin{proof}
    By Corollary \boxed{\ref{3.8}} and Theorem \boxed{\ref{2.17}}. The factors contain $v_i$ come from the case $a_1=0$ but $a_2\neq0$ and the factors contain $u_i$ come from $a_1\neq0$ in Theorem \boxed{\ref{2.17}}. 
\end{proof}

\section{$\mathbb Z_p^\times$-extensions of local field}\label{sec4}

Let $k$ be a finite field of characteristic $p$ and $K=k(\!(T)\!)$ be a local field over $k$. The main goal of this section is to find the valuation of different for the $(\mathbb Z/p^{n+1}\mathbb Z)^\times$-extension of $K$ (Corollary \boxed{\ref{4.13}}).  Recall in Corollary \boxed{\ref{3.8}}, when $p>2$, we can write $x\equiv [cT^d]\cdot\prod_{(i,p)=1}u_i^{k_i}\bmod\mathfrak qW(K)^\times$ with $[\quad]$ be Teichm\"uller lifting and $u_i=1+V[T^i]$.  In Section \ref{sec5.1} we discuss the case $\mathfrak qy=[cT^d]$, and in Section \ref{5.2} we discuss the case $\mathfrak qy=u_i$. Recall that we denote $L_n=K(\mathfrak q^{-1}x)$ for $x\in W_{n+1}(K)^\times$.  In Section \ref{5.3} we will compute the discriminant of field extension $L_{n}/K$. 
\subsection{Equation $\mathfrak qy=[cT^d]$}\label{sec5.1}
From now on, we focus on the case when $p>2$. We first consider the case of equation $\mathfrak qy=[cT^d]$. 
\begin{lemma}\label{4.2}
Let $ K=k(\!(T)\!),u\in K^\times/(K^\times)^n$ with $v_T(u)=i\in\mathbb Z.$ Write $u=cT^i$ for some $c=c_0+c_1T+c_2T^2+\cdots\in k[\![T]\!]^\times$, $0< i\leq n-1$. Then $f(x)=x^n-u$ is irreducible in $K[x]$ if and only if $c_0\notin(k^\times)^m$ for each $m>1$ satisfies $m\mid(n,i)$.
\end{lemma}
\begin{proof}
By Hensel's lemma, one has $c\notin(k[\![T]\!]^\times)^m$ equivalent to $c_0\notin(k^\times)^m$. Hence, we can reduce $u=cT^i$ to $u=c_0T^i$ for $c_0\equiv c\bmod T$. Then we can prove the statement by considering the cyclotomic extension of $k$.
\end{proof}

\begin{lemma}\label{4.3}
Let $L=K(y)$, $f(x)=x^n-u$ is irreducible over $K$ with $v_T(u)=i$, $y$ is a root of $f(x)=0$, 
\begin{itemize}
    \item if $(n,i)=1$, then $\Gal(L/K)\cong\mathbb Z/n\mathbb Z$ and $L/K$ is totally ramified.
     \item if $(n,i)\neq1$, then $L/K$ is totally ramified of degree $e$ and unramified of degree $f$, where $f$ is the minimum number such that $(i,n)\mid p^f-1$ and $e=\frac{n}{f}$.
     \end{itemize}
\end{lemma}

\begin{proof}
By Corollary \boxed{\ref{4.2}}, let $y^{n}=cT^i$, one has $nv_{\varpi_1}(y)=iv_{\varpi_1}(T)$. If $(n,i)=1$, then $v_{\varpi_1}(T)=n$. Hence $L/K$ is totally ramified.

If $(n,i)\neq1$, we may choose $y=\sqrt[n]{cT^i}$. Let $(n,i)=n_0$, $K'=K(c^{\frac{1}{n_0}})$, we have $K'$ be a midfield of field extension $L/K$ with $[K':K]=n_0,[L:K']=\frac{n}{n_0}$. By the first case, we have $L/K'$ is totally ramified as $(\frac{n}{n_0},i)=1$. For $K'/K$, as $(n_0,p)=1$, so $K'/K$ is unramified of degree $f$ where $f$ is the minimum number satisfied $n_0\mid p^f-1$. Hence $L/K$ is unramified of degree $f$ and totally ramified of degree $\frac{n}{f}$. 
\end{proof}
\subsection{Equation $\mathfrak qy=u_i$}\label{5.2}
Now we consider the case $\mathfrak qy=u_i$ for $p>2$ with $y\in W(K^{\sep})^\times$. More precisely, by Section \ref{sec5.1}, we may assume $y_0=1$ and write $y_n^p=P_n$ with $$P_0=y_0=1$$
\begin{align}
P_{n}=y_n+(T^i)^{p^{n-1}}y_{n-1}^p-\sum_{k=1}^{n-1}\frac{1}{p^{n-k}}[P_{k}^{p^{n-k}}-y_k^{p^{n-k}}-(T^i)^{p^{n-1}}y_{k-1}^{p^{n-k+1}}]\tag{**}
\end{align}
for all $n\geq1$. 

Let $L_n=K(\mathfrak q^{-1}x)$ for some $x\in W_{n+1}(K)^\times$. Let $\varpi_n$ be a prime of $L_n$ lies above $T$,$n=0,1,\cdots$. Let $u_i=1+V[T^i]$ be as in Theorem \boxed{\ref{037}}.
\begin{lemma}
    
{\label{39}}
The equation $\mathfrak qy=u_i$ has solutions in $W(K)^\times$ if and only if $i>0$ or $i<0$ but $p\mid i$.
\end{lemma}}

\begin{proof}
If $i>0$, one can yield a truncated result on $W_n(K)^\times$ by induction and the Hensel's lemma. The result follows from passing the limit.

If $i\leq 0$, we will show $\mathfrak qy=u_i$ does not admit a solution in $W(K)^\times$ for all $(i,p)=1,i<0$ and $i=0$. If $i=0$, then $y_1^p-y_1\equiv 1\bmod T$ with $v_{\varpi_1}(y_1)=0$. Since the constant term of $y_1$ always lies in $k$, the constant term of $y_1^p-y_1$ is $0$, which contradicts $v_{\varpi_1}(y_1)=0$. If $i<0$, consider the equation $y_{1}^p-y_{1}-T^iy_{0}^p=0$. Since $v_{\varpi_1}(y_{0})=0$ as $y_0$ is $1$ by $(**)$, we have $i=v_{\varpi_1}(y_{1}^{p}-y_{1})=pv_{\varpi_1}(y_{1})$. Then $y_{1}^p-y_{1}-T^iy_{0}^p=0$ has a solution in $K$ if and only if $p\mid i$ when $i<0$.
\end{proof}

We focus on case $i<0$ and $(i,p)=1$ in the rest of Section \ref{5.2} and we will discuss case $i=0$ in Section \ref{5.3}.

\begin{lemma}\label{4.8}
    For each $n>1$, there exists $f_n\in L_{[n-1]},g_n\in L_{[1]}$ such that $$y_n^p-y_n=f_n^p-T^{(p^{n-1}-1)i}y_1+g_n$$ and $v_{\varpi_1}(g_n)>v_{\varpi_1}(T^{(p^{n-1}-1)i}y_1)$.
\end{lemma}
\begin{proof}
    For a fixed $n$, if $k=1$, we have $v_{\varpi_1}(\frac{1}{p^{n-1}}(P_1^{p^{n-1}}-y_1^{p^{n-1}}-T^{p^{n-1}i}))=v_{\varpi_1}(T^{(p^{n-1}-1)i}y_1)$ as $v_{\varpi_1}(T^i)<v_{\varpi_1}(y_1)$.
    
    When $k>1$, since all the coefficients in the $(**)$ lie in $k$, then by binomial expansion and notice the fact that for all $k$, we have $P_k-T^{p^{k-1}i}y_{k-1}^p=y_k-\sum_{j=1}^{k-1}\frac{1}{p^{k-j}}[P_j^{p^{k-j}}-y_i^{p^{k-j}}-T^{p^{k-1}i}y_{j-1}^{p^{k-j+1}}]$, there exists $f_k\in L_{[k-1]}$ such that $$P_k=y_k+f_k^p+(\sum_{m=1}^{p^{k-1}-1}c_mT^{mi}y_1^{p^{k-1}-m})$$
    for some $c_m\in k$. Moreover, by $(**)$, $c_{p^{k-1}-1}=-1$ and the claim follows.
\end{proof}

\begin{lemma}\label{4.9}
    Let $1<k\leq n$. If $z_k:=y_k-f_k\in L_{[n]}$ satisfies $z_k^p-z_k=-T^{(p^{k-1}-1)i}y_1+f_k$ for $k=2,\cdots,n$, then there exists $y_k'=z_k+s_k^p$ for some $s_k\in L[T,y_1',y_2',\cdots,y_{k}']$ and $y_k'^p-y_k'=-T^{(p^{k-1}-p^{k-2})i}y_{k-1}'+t_k$ and $v_{\varpi_{k-1}}(t_k)>v_{\varpi_{k-1}}(T^{(p^{k-1}-p^{k-2})i}y_{k-1}')$ with $t_k\in L[T,y_1',y_2',\cdots,y_{k}']$.
\end{lemma}
\begin{proof}
    We show the lemma by induction. Let $y_1=y_1'$. When $n=2$,  one can write $(y_2-T^iy_1)^p-(y_2-T^iy_1)=-\sum_{k=1}^{p-1}T^{ki}y_1^{p-k}+T^iy_1$ and let $y_2'=z_2=y_2-T^iy_1$.
    
    Assume there exists $s_{k-1},t_{k-1}\in L[T,y_1',y_2',\cdots,y_{k-1}']$ such that $$-T^{-i}y_1+s_{k-1}^p=-T^{-p^{k-2}i}y_{k-1}'+t_{k-1}$$ with $v_{\varpi_{k-1}}(t_{k-1})>v_{\varpi_{k-1}}(T^{-p^{k-2}i}y_{k-1}')$ and therefore $pv_{\varpi_{k}}(y_k)=v_{\varpi_k}(T^{(p^{k-1}-p^{k-2})i}y_{k-1}')$.
    
    Then
    \begin{eqnarray*}
        z_{k+1}^p-z_{k+1}&=&-T^{(p^{k}-1)i}y_1+f_k\\
        &=&-T^{(p^k-p^{k-2})i}y_{k-1}'+T^{p^ki}t_{k-1}-T^{p^ki}s_{k-1}^p+f_k\\
        &=&T^{(p^{k}-p^{k-1})i}(y_k'^p-y_k')+t'_{k}\\
        &=&(T^{(p^{k-1}-p^{k-2})i}y_k')^p-T^{(p^k-p^{k-1})i}y_k'+t'_{k}
    \end{eqnarray*}
    for some $t_{k}'\in L[T,y_1',y_2',\cdots,y_k']$. 

    Let $t_{k}=t_{k}'+{s'_{k}}^p$ for some $s_k'\in L[T,y_1',\cdots,y_k']$ such that $v_{\varpi_k}(t_k)>v_{\varpi_k}(t_k')$. Assume the leading term in $t_{k}'$ is of the form $c_{m,k}T^m\prod_{j=1}^{k}y_j'^{a_j}$ for $a_j=0,1,\cdots,p^{k}-1,c_{m,k}\in k$. By $(**)$ and Lemma \boxed{\ref{4.8}}, one can easily see $m=(p^{k}-\sum_{j=1}^{k}a_jp^{j-1})i$. Notice $$v_{\varpi_{k+1}}(y_k')<v_{\varpi_{k+1}}(y_k'+T^{(p^{k-1}-p^k)i}y_{k+1}'^p)$$ for all $k<n-1$, hence one can always assume $a_{k}\neq0$. Since $v_{\varpi_j}(T^{p^{j-1}i})<v_{\varpi_j}(y_j')$ for all $j=1,2,\cdots,n$, one has $v_{\varpi_k}(T^{(p^k-p^{k-1})i}y_k')>v_{\varpi_k}(c_{m,k}T^m\prod_{j=1}^{k}y_j'^{a_j})$ for all $a_{k}\neq0$.
\end{proof}
\begin{remark}
    If $p=2$, by Corollary \boxed{\ref{3.8}}, it suffice to consider equations $\mathfrak qy=u_i$ and $\mathfrak qy=v_i$, but both cases are the same as $\mathfrak qy=u_i$ for $p>2$ by previous lemmas.
\end{remark}
\subsection{Ramification information of the $\mathbb Z_p^\times$-towers}\label{5.3}
Recall we denote $L_n$ as $L_{[n]}$ if $\Gal(L_n/K)\cong(\mathbb Z/p^{n+1}\mathbb Z)^\times$. We first discuss the ramification information of $L_{[n]}/K$ and then compute its discriminant. 

Notice we have the facts that when $
\mathfrak q a=u_0$, then $L_{[n]}/L_{[0]}$ is totally ramified of degree $p^{n}$. By Proposition \boxed{\ref{39}} and Lemma \boxed{\ref{4.9}},  if $y$ satisfies an equation $\mathfrak qy=u_i$ for some $i<0$ with $(i,p)=1$, then $L_{[n]}/L_{[0]}$ is totally ramified of degree $p^{n}$.

Let $p>2$. Let $x\equiv [cT^d]\cdot\prod_{i}u_{i}^{k_i}\in W_n(K)/\mathfrak qW_n(K)^{\times}$ be a reduced form. Recall $L_{n-1}=K(\mathfrak q^{-1}x)$. Let $e_1$ as defined in Lemma \boxed{\ref{4.3}}, 
$n_c=\max\{m:\pi_{m}(u_{i_i})\neq 1,\forall u_{i}\}-1$, 
$n_1=\min_i\{v_p(k_i):i=0\}$,
$n_2=\min_i\{v_p(k_i):i<0\}$ or $n$ if such $i$ does not exists, and $n_u=\max\{n_2-n_1,0\}$, then we have 
\begin{corollary}
    If $y$ satisfies equation $\mathfrak qy=x$, then the cyclic extension $L_{[n]}/K$ is unramified of degree $\frac{p-1}{e_1}\cdot p^{n_u}$ and totally ramified of degree $e_1\cdot p^{n-n_2}$, where $e_1$ is the ramification index of $L_{[0]}/K$. 
\end{corollary}
\begin{proof}
    By Lemma \boxed{\ref{4.3}} and the fact $L_{[n]}/L_{[0]}$ is totally ramified if $y$ satisfies an equation $\mathfrak qy=u_i$ for some $i<0$ with $(i,p)=1$.
\end{proof}

\begin{remark}
We call a tower geometric means that it cannot be written as a union of two disjoint algebraic curves. One can see if $n_c=0$, we have this tower be a geometric tower.
\end{remark}
\begin{lemma}\label{5.14}
    Consider equation $\mathfrak qx=cT^d$ for $-p+1\leq d<0$. Let $1\leq m\leq p-1$ be the maximum integer such that $c\in\mathbb F_p^m$, then
    $$v_{\varpi_{0}}(\mathfrak D_{L_{0}/K})=p-1-{(-d,m,p-1)},$$
    where $(a,b,c)$ denotes the gcd of $a,b,c$.
    
\end{lemma}
\begin{proof}
    By $\mathfrak qy_0=cT^d$, $y_0$ satisfies $y_0^{p-1}=cT^d$. Then the ramification index of $L_{0}/K$ is $\frac{p-1}{(-d,m,p-1)}$ and $v_{\varpi_0}(T)=\frac{p-1}{(-d,m,p-1)},v_{\varpi_0}(y_0)=\frac{d}{(d,m,p-1)}$. Then the order of the inertia group of $T$, we denoted it as $e_T$, is $\frac{p-1}{(-d,m,p-1)}$. The rest follows from the Riemann-Hurwitz formula \cite[Theorem 3.4.13]{SH}.
\end{proof}
\begin{proposition}\label{4.12}
    If $x\equiv u_i^{k_i}\bmod \mathfrak qW_n(K)^\times$ with $k_i\neq0$, then $$v_{\varpi_n}(\mathfrak D_{L_{n}/K})=\frac{|i|}{p+1}p^{2(n-n_u)}+p^{n-n_u}+\frac{1}{p+1}i-1.$$
\end{proposition}
\begin{proof}
    We assume $y_0=1$ for simplicity. When $n=1$, $v_{\varpi_1}(y_1)=i$. When $n>1$ $v_{\varpi_n}(y_n')=[\sum_{j=1}^{n-1-n_u}(p^{2j}-p^{2j-1})+1]i$. Then there exists $a,b\in\mathbb Z$ such that $\pi_n=y_n'^aT^b$. If $n>n_u+1$, we have the discriminant of $L_n/L_{n-1}$
    \begin{eqnarray*}
    v_{\varpi_n}(\mathfrak D_{L_{n}/L_{n-1}})&=&\sum_{\sigma\in\Gal(L_{n}/L_{n-1}),\sigma\neq\id}v_{\varpi_n}(\sigma(\varpi_n)-\varpi_n)\\
    &=&\sum_{\sigma\in\Gal(L_{n}/L_{n-1}),\sigma\neq\id}[bv_{\varpi_n}(T)+v_{\varpi_n}(\sigma(y_n'^a)-y_n'^a)]\\
    &=&b(p-1)v_{\varpi_n}(T)+(p-1)(a-1)v_{\varpi_n}(y_n')\\
    &=&p-1-v_{\varpi_n}(y_n')\\
    &=&p-1-(p-1)[\sum_{j=1}^{n-n_u-1}(p^{2j}-p^{2j-1})+1]i\\
    &=&p-1-(p-1)i-(p-1)^2\cdot\frac{1-p^{2(n-n_u-1)}}{1-p^2}i\\
    &=&-\frac{p^{2(n-n_u)}-p^{(2n-2n_u-1)}+p-1}{p+1}i+(p-1).
    \end{eqnarray*}
Similarly, $v_{\varpi_{n_u+1}}(\mathfrak D_{L_{n_u+1}/L_{n_u}})=p-1-(p-1)v_{\varpi_{n_u}}(y_{n_u+1}')=(p-1)(1-i)$.
Hence
    \begin{eqnarray*}
    v_{\varpi_n}(\mathfrak D_{L_{n}/K})&=&v_{\varpi_n}(\prod_{k=0}^n\mathfrak D_{L_{k}/L_{k-1}})\\
    &=&\sum_{k=n_u}^{n-1}v_{\varpi_n}(\mathfrak D_{L_{k+1}/L_{{k}}})\\
    &=&\sum_{k=n_u}^{n-1}p^{n-k-1}v_{\varpi_{k+1}}(\mathfrak D_{L_{k+1}/L_{k}})\\
    &=&\frac{-p^{2(n-n_u)}+1}{p+1}i+p^{n-n_u}-1.
    \end{eqnarray*}
\end{proof}
\begin{proposition}\label{4.13}
If $x\equiv [cT^d]\cdot\prod_iu_i^{k_i}\bmod\mathfrak qW_n(K)^\times$ as in Corollary \boxed{\ref{3.8}} and there exists $i$ such that $k_i\neq0$. Then $$v_{\varpi_n}(\mathfrak D_{L_{n}/K})=p^{n-n_u}(p-{(-d,m,p-1)})+\max_{i}\{\frac{|i|}{p+1}p^{2(n-n_u)}+\frac{1}{p+1}i\}-1$$
with $m$ as defined in Lemma \boxed{\ref{5.14}}.
\end{proposition}
\begin{proof}
   By Proposition \boxed{\ref{4.12}}, consider field extension generated by solution of equation $\mathfrak qy=[cT^d]\cdot u_i^{k_i}$ for a fixed $i$, then
    \begin{eqnarray*}
    v_{\varpi_n}(\mathfrak D_{L_{n}/K})&=&v_{\varpi_n}(\prod_{k=0}^n\mathfrak D_{L_{k}/L_{k-1}})\\
    &=&\sum_{k=n_u}^{n-1}p^{n-k-1}v_{\varpi_{k+1}}(\mathfrak D_{L_{k+1}/L_{k}})+p^{n-n_u}v_{\varpi_0}(\mathfrak D_{L_{0}/K})\\
    &=&p^{n-n_u}(p-1-(-d,m,p-1))+p^{n-n_u}+\frac{-p^{2(n-n_u)}+1}{p+1}i-1.
    \end{eqnarray*}
Notice the valuation of $y_i$ in $(**)$ is determined by the minimum value of $i$, we have the desired formula. 
\end{proof}

\begin{remark}
    For the case $p=2$, one can use a similar technique in Lemma \boxed{\ref{4.8}} and Lemma \boxed{\ref{4.9}} for equations $\mathfrak qx=u_i$ and $\mathfrak qx=v_i$. Once the representation of $y_n'\in L_{[n]}$ satisfies $p\nmid v_{\varpi_n}(y_n')$, we can repeat the process in Proposition \boxed{\ref{4.12}} to find the valuation of the different.
\end{remark}
\begin{remark}
    In Proposition \boxed{\ref{4.13}}, we fix a typo in \cite[Proposition 5.1]{MR3754335}, the product of $i$ should start from $1$ to $n-n_{\mathfrak p,0}$ other than from $n_{\mathfrak p,0}+1$ to $n$ and hence the result of \cite[Proposition 5.3]{MR3754335} also needs to be modified. However, the typo does not affect the result in \cite[Proposition 5.5]{MR3754335}.
\end{remark}
\section{$\mathbb Z_p^\times$-covers of projective line}\label{sec5}
In this section, we will prove Theorem \boxed{\ref{1.5}}.

Let $k$ be a finite field of characteristic $p>2$ and $K$ be a function field over $k$ with $k$ be the full constant field. Let $x=(x_0,x_1,\dots)\in W(K)^\times$, and we assume that 
$x_1\neq0$ for simplicity (see Corollary \boxed{\ref{2.10}}) in this section. 
$L_{n}=K(y_0,y_1,\cdots,y_{n})$ with $\mathfrak q(y_0,y_1,\cdots,y_n)=(x_0,x_1,\cdots,x_{n})$ and $L_{-1}=K,L_{\infty}=K(y_0,y_1,\cdots)$. Then we have a tower of field extensions $L_{-1}\subseteq L_{0}\subseteq\cdots\subseteq L_{\infty}$ with $\Gal(L_{n}/L_{-1})$ be a quotient group of $(\mathbb Z/p^{n+1}\mathbb Z)^\times$. In this section, we compute the genus growth by applying Proposition \boxed{\ref{4.13}} for a $\mathbb Z_p^\times$ cover of the projective line $\mathbb P^1$ over a finite field $k$ in Example \boxed{\ref{6.4}}.

\subsection{Genus growth formula for global field}\label{s5.1}
Let $\mathfrak p$ be a place of $K$. Let $K_\mathfrak p$ be the localization of $K$ at $\mathfrak p$, and $k_\mathfrak p$ be its residue field with $\varpi_{\mathfrak p}$ be a uniformizer in $k_{\mathfrak p}$. Let $\widehat{K_\mathfrak p}$ be the completion of $K_\mathfrak p$, then $\widehat{K_\mathfrak p}\cong k_\mathfrak p(\!(\varpi_\mathfrak p)\!)$ \cite[Theorem 15]{5953cc44-5a9f-317f-90f3-83c005ff4b88}. 

Let $x\in W(\widehat{K_\mathfrak p})^\times$, $u_{\mathfrak p,i}=1+V[\frac{1}{\varpi_\mathfrak p}]^i$. By Corollary \boxed{\ref{3.8}}, we have 
$$x\equiv [c\varpi_\mathfrak p^d]\cdot\prod_{i_{\geq0}}u_{\mathfrak p,i}^{k_i}\bmod\mathfrak qW(\widehat{K_\mathfrak p})^\times$$
with $c\in k^\times,(i,p)=1$ or $i=0$, $0\leq d<p-1,k_i\in \mathbb Z_p$ and $v_p(k_i)\rightarrow 0$ as $i\rightarrow \infty$. From now on, we assume $(d,p-1)=1$ for simplicity. 

Let $g_{n}$ be the genus of $L_{n}$, which is the genus corresponding to the algebraic curve defined by $L_{n}$ over the intersection of the algebraic closure of $k$ and $L_{n}$. Let $n_c$ be the maximum integer such that $L_{n_c}/L_{0}$ is a constant field extension. Let $\varpi_{\mathfrak {P_n}}$ be a prime of the residue field of $L_{n}$ lies over $\varpi_{\mathfrak p}$. Then we have a relation between the genus of $L_{n}$ and $K$,
\begin{theorem}\label{5.1}
    For $n\in\mathbb Z_{\geq0}$, we have genus formula
    $$\frac{p-1}{e_1}\cdot p^{\min\{n_c,n\}}(2g_n-2)=(p-1)p^{n}(2g_{-1}-2)+\prod_{\mathfrak p}v_{\varpi_{\mathfrak P_n}}(\mathfrak D_{L_{n}/K,\mathfrak p}),$$
    where $e_1$ is the interia degree of $L_0/K$.
\end{theorem}
\begin{proof}
    Let $l_n$ be the full constant field of $L_{n}$, $\mathfrak D_{L_{n}/K}$ be the different of $L_{n}/K$ and $\mathfrak D_{L_{n}/K,\mathfrak p}$ be the local different at the place $\mathfrak p$. Then by Riemann-Hurwitz formula,
    $$[l_n:k](2g_n-2)=[L_{n}:K](2g_{-1}-2)+\deg_{k}\mathfrak D_{L_{n}/K}$$
    The Theorem follows from $[l_n:k]=\frac{p-1}{e_1}\cdot p^{\min\{n_c,n\}}$, $\Gal(L_{n}/K)\cong(\mathbb Z/p^{n+1}\mathbb Z)^\times$.
\end{proof}
\begin{proposition}\label{4.18}
    If $L_{\infty}/K$ is a geometric $\mathbb Z_p^\times$ extension, then the following are equivalent,

    i) There exists $s\in\mathbb Z_{>0}$ with $(s,p)=1$ and $m,t\in\mathbb Z_{\geq0}$ with $s\equiv t\bmod p-1$, when $n$ is large enough, $$\max_i\{ip^{-v_p(k_i)}:(i,p)=1\}=\frac{1}{p-1}(\frac{s}{p^m}-\frac{t}{p^{n-1}}).$$


    ii) There exists $a,b,c\in\mathbb Q$ such that for all $N\in\mathbb Z_{>0}$, if $n>N$, then $$g_n=ap^{2n}+bp^n+c.$$
\end{proposition}
\begin{proof}
    Suffice to show the statement holds for local field extension. By Theorem \boxed{\ref{5.1}}, $$\frac{p-1}{e_1}(2g_n-2)=-2(p-1)p^{n}+p^{n-n_u}(p-1-\frac{p-1}{e_1})+\max_i\{\frac{i}{p+1}p^{2(n-n_u)}-\frac{1}{p+1}i\}+p^{n-n_u}-1,$$ 
    then $$2(g_n-g_{n-1})=e_1[\max_i\{p^{2n-2n_u-2}i\}-2(p^{n}-p^{n-1})+p^{n-n_u-1}(p-\frac{p-1}{e_1})].$$
    One can see from the difference of $g_n-g_{n-1}$ is of the form $ap^{2n}+bp^n$ for some $a,b\in\mathbb Q$.

    When $n$ is large enough, write $\frac{2(g_n-g_{n-1})}{p^n}=p^{n}s_n+r$ and $p^ns_n=\alpha p^n-\beta$. Then $p^ns_n-p^{n-1}s_{n-1}=\alpha(p-1)p^{n-1}$ for all $n$. Hence we can write $\alpha=\frac{s}{(p-1)p^m}$ for some $m\in\mathbb Z_{\geq0}$ and $p\nmid s$. Since $i$ is an integer, we have $\beta=\frac{t}{p-1}$ for some $t\in\mathbb Z$ and $s\equiv t\bmod p-1$. The result follows from $n_u=\{n_2-n_1,0\}$ and $n_1=\min_i\{v_p(k_i):i=0\},n_2=\min_i\{v_p(k_i):i<0\}$ as defined in  Section \ref{sec4}.
\end{proof}

\subsection{Genus growth formula for function field}
Let $k$ be a finite field of characteristic $p>2$, $K=k(X)$ be a function field of $k$. Let $\bar k$ be the algebraic closure of $k$ and $K'=\bar k(X)$. For each $x\in\bar k$, let $\varpi_x=X-x$ and $\varpi_\infty=\frac{1}{X}$. Then we can do localization of $K$ at $\varpi_x$ as in  Section \ref{s5.1}. 

\begin{lemma}
    Let $u_{x,i}=1+V[\frac{1}{\varpi_x}]^i$ with $i\geq0$. There exists a unique representation of $a\in W(K)^\times/\mathfrak qW(K)^\times$,
    $$a=\prod_{x\in\mathbb P^1(\bar k)}[c\varpi_x^{d_x}]\cdot\prod_i u_{x,i}^{k_{x,i}}$$
    for $c\in k^\times, 0\leq d_x<p-1, (i,p)=1$ or $i=0$, $k_{x,i}\in\mathbb Z_p$ and for a fixed $c\in\mathbb Z_{\geq0}$, there are only finitely many $k_{x,i}$ such that $v_p(k_{x,i})<c$.
\end{lemma}
\begin{proof}
    By Corollary \boxed{\ref{3.7}} and Proposition \boxed{\ref{4.18}}.
\end{proof}

Now we see two examples, the first one is a nondegenerate $\mathbb Z_p^\times$-cover of $\mathbb P^1$ and the second is a degenerate $\mathbb Z_p^\times$-extension, which is equivalent to the result in the $\mathbb Z_p$-cover of $\mathbb P^1$.
\begin{example}\label{6.4}
    Let $u_i=1+V[\frac{1}{X}]^i$. Consider a unit root $\mathbb Z_p^\times$ extension of $K$, generated by $\mathfrak q^{-1}x$ with $$x=[\frac{1}{X}]\cdot\prod_{(i,p)=1~or~i=0}^d u_i^{k_i}\in W(K)^\times$$ for $k_i\in\mathbb Z_p$. Then this tower is totally ramified at $\infty$. We can further assume $(d,p)=1$. By Proposition \boxed{\ref{4.13}}, the different of this $\mathbb Z_p^\times$-tower at place $\infty$ is 
    $$\mathfrak D_{L_{[n]}/K,\infty}=\frac{dp^{2n}}{p+1}+p^n(p-1)-\frac{d}{p+1}-1$$ 
    Hence by Theorem \boxed{\ref{5.1}},
    \begin{eqnarray*} 
    2g_n-2&=&-2p^{n}(p-1)+\frac{dp^{2n}}{p+1}+p^n(p-1)-\frac{d}{p+1}-1\\
    &=&\frac{d}{p+1}p^{2n}-p^{n}(p-1)-\frac{d}{p+1}-1.
    \end{eqnarray*}
\end{example}
\begin{example}
    Keep notation as in the previous example. If $x=\prod_{(i,p)=1}^d u_i^{k_i}\in W(K)^\times$, which is a degenerate $\mathbb Z_p^\times$-tower as $L_{[1]}=L_{[0]}$. Then by Proposition \boxed{ \ref{4.13}} its genus satisfies 
    $$2g_n-2=\frac{d}{p+1}p^{2n}-p^{n}-\frac{d}{p+1}-1.$$
    This implies the same unit root Artin-Schreier-Witt $\mathbb Z_p$-towers in \cite[Example 5.10]{MR3754335}.
\end{example}
\bibliographystyle{amsplain}
\bibliography{ref.bib}

\providecommand{\bysame}{\leavevmode\hbox to3em{\hrulefill}\thinspace}
\providecommand{\MR}{\relax\ifhmode\unskip\space\fi MR }
\providecommand{\MRhref}[2]{%
  \href{http://www.ams.org/mathscinet-getitem?mr=#1}{#2}
}
\providecommand{\href}[2]{#2}
\begin{thebibliography}{10}

\bibitem{5953cc44-5a9f-317f-90f3-83c005ff4b88}
I.~S. Cohen, \emph{On the structure and ideal theory of complete local rings}, Transactions of the American Mathematical Society \textbf{59} (1946), no.~1, 54--106.

\bibitem{Koblitz1984}
Neal Koblitz, \emph{p-adic numbers, p-adic analysis, and zeta-functions}, Springer New York, New York, NY, 1984.

\bibitem{kosters2016arithmeticzpextensions}
Michiel Kosters and Daqing Wan, \emph{On the arithmetic of {$\Bbb Z_p$}-extensions}, 2016.

\bibitem{MR3754335}
Michiel Kosters and Daqing Wan, \emph{Genus growth in {$\Bbb Z_p$}-towers of function fields}, Proc. Amer. Math. Soc. \textbf{146} (2018), no.~4, 1481--1494. \MR{3754335}

\bibitem{MR1878556}
Serge Lang, \emph{Algebra}, third ed., Graduate Texts in Mathematics, vol. 211, Springer-Verlag, New York, 2002. \MR{1878556}

\bibitem{morris_1977}
Sidney~A. Morris, \emph{Pontryagin duality and the structure of locally compact abelian groups}, London Mathematical Society Lecture Note Series, Cambridge University Press, 1977.

\bibitem{MR1697859}
J\"{u}rgen Neukirch, \emph{Algebraic number theory}, Grundlehren der mathematischen Wissenschaften [Fundamental Principles of Mathematical Sciences], vol. 322, Springer-Verlag, Berlin, 1999, Translated from the 1992 German original and with a note by Norbert Schappacher, With a foreword by G. Harder. \MR{1697859}

\bibitem{MR0554237}
Jean-Pierre Serre, \emph{Local fields}, Graduate Texts in Mathematics, vol.~67, Springer-Verlag, New York-Berlin, 1979, Translated from the French by Marvin Jay Greenberg. \MR{554237}

\bibitem{SH}
Henning Stichtenoth, \emph{Algebraic function fields and codes}, Springer-Verlag, Berlin, 2009.

\bibitem{Thomas}
Lara Thomas, \emph{Arithmétique des extensions d'{A}rtin-{S}chreier-{W}itt}, PhD. thesis (2005).

\end{thebibliography}

\end{document}